\documentclass[10pt,a4paper,reqno]{amsart}
\usepackage{relsize,exscale}
\usepackage{hyperref}
\newcommand*{\affaddr}[1]{#1}
\newcommand*{\affmark}[1][*]{\textsuperscript{#1}}
\usepackage{graphics,layout,indentfirst,amsmath,amsthm,amssymb}
\newtheorem{theorem}{Theorem}

\theoremstyle{theorem}

\newtheorem{lemma}[theorem]{Lemma}

\numberwithin{equation}{section}
\makeatletter
\@namedef{subjclassname@2010}{2010 Mathematics Subject
Classification} \makeatother
\begin{document}
\title[On the summability of Random Fourier--Jacobi Series]{On the summability of Random Fourier--Jacobi Series}
\author[P. Maharana, S. Sahoo]
{P\lowercase{artiswari} M\lowercase{aharana}\affmark[1] \lowercase{and}
S\lowercase{abita} S\lowercase{ahoo}\affmark[2]\\
\affaddr{\affmark[1] D\lowercase{epartment} \lowercase{of} M\lowercase{athematics}, S\lowercase{ambalpur} U\lowercase{niversity}, O\lowercase{disha}, I\lowercase{ndia}}\\
\affaddr{\affmark[2] D\lowercase{epartment} \lowercase{of} M\lowercase{athematics}, S\lowercase{ambalpur} U\lowercase{niversity}, O\lowercase{disha}, I\lowercase{ndia}.
}\\
{\affmark[1] \lowercase{partiswarimath1@suniv.ac.in} }\\
{\affmark[2] \lowercase{sabitamath@suniv.ac.in} } \\
 }

\begin{abstract}
This article is a study on the summability of random Fourier--Jacobi series of some functions in different spaces.
We consider the random series
$
\sum_{n=0}^\infty a_nA_n(\omega)p_n^{(\gamma,\delta)}(y),
$
where $p_n^{(\gamma,\delta)}(y),\gamma,\delta>-1$ are orthonormal Jacobi polynomials, the scalars $a_n$
are Fourier--Jacobi coefficients of a function $f$ and
the random variables $A_n(\omega)$
are Fourier--Jacobi coefficients of the symmetric stable process $X(t,\omega)$ of index $\alpha \in [1,2].$
It is established that the random Fourier--Jacobi series is $\Theta$--summable in probability, if $a_n$ are
the Fourier--Jacobi coefficients of function $f$ in the space $C_{[-1,1]}^{(\eta,\tau)}.$
The Ces{\'a}ro $(C,\phi),\phi \geq1$ summability of random Fourier--Jacobi series is shown, for the symmetric stable process $X(t,\omega)$ of index $\alpha \in [1,2]$ under different conditions on the parameters
 $\gamma,\delta,\eta$ and $\tau.$
The other cases of summability, such as Riesz, Rogosinski, etc., are also discussed.
Further, the N{\"o}rlund summability, generalized N{\"o}rlund summability, and lower triangular summability of random Fourier--Jacobi series are proved if $a_n$ are the Fourier--Jacobi coefficients of a function $f \in L_{[-1,1]}^{1,(\gamma,\delta)},$ and $A_n(\omega)$ are associated with the symmetric stable process $X(t,\omega)$ of index one. It is observed that the conditions on the parameters $\gamma,\delta$ differ from that of the conditions on $\gamma,\delta$
for the Fourier--Jacobi series of functions $f$ in $L_{[-1,1]}^{1,(\gamma,\delta)}.$

{\bf 2020 MSC Classification}-: 60G99, 40G15.

{\bf Key words}: Ces{\'a}ro summability; Convergence in probability; Fourier-Jacobi series; Jacobi polynomials; Lower triangular summability; N{\"o}rlund summability; Random variables; Stochastic integral; Symmetric stable process.
  
\end{abstract}
\maketitle{}
\section{Introduction}
The foremost aspect in the study of Fourier analysis is the convergence of Fourier series of functions.
In the general case, the convergence of the Fourier series of all functions is not always possible.
However, using some summability methods, convergence is possible in some sense, like in Ces{\'a}ro, Riesz, N{\"o}rlund, etc.
The summability of Fourier series in orthogonal polynomials is investigated by many researchers. Also, the weighted convergence of certain sums in the
Fourier--Jacobi series was investigated by Lubinsky, Totik \cite{LT}, and many others.

Further, Nayak, Pattanayak, Mishra \cite{NPM}, Towghi \cite{T1}, and others studied the summability of random Fourier series.
Motivated by the works of Liu and Liu \cite{LL,LL1}, we investigated the convergence of random Fourier--Jacobi series of functions $f$ in various continuous spaces \cite{PS} and the space $L_{[-1,1]}^p$  \cite{PS1}.

In this article, we have studied the summability of random Fourier series
\begin{equation}\label{1.11}
\sum_{n=0}^\infty a_n A_n(\omega) p_n^{(\gamma,\delta)}(y)
\end{equation}
in orthonormal Jacobi polynomial $p_n^{(\gamma,\delta)}(y),$ where $a_n$ are scalars, $A_n(\omega)$ are random variables.
The scalars $a_n$ are the Fourier--Jacobi coefficients of a function $f$ defined as
\begin{equation}\label{1.2}
a_n:=\int_{-1}^1 f(t) p_n^{(\gamma,\delta)}(t)\rho^{(\gamma,\delta)}(t)dt,\; \gamma,\delta>-1.
\end{equation}
The functions $f$ chosen here are from the continuous space $C_{[-1,1]}^{(\eta,\tau)}$ and the
space $L_{[-1,1]}^{p,(\gamma,\delta)}.$
In relation to the space $C_{[-1,1]}^{(\eta,\tau)},$ the $\Theta$--summability, Ces{\'a}ro $(C,\phi)$ summability, and some other summability of random series (\ref{1.11}) is established. The N{\"o}rlund summability, generalized N{\"o}rlund summability, and lower triangular summability of random series (\ref{1.11}) are established in connection to the space $L_{[-1,1]}^{1,(\gamma,\delta)}.$

The class of continuous functions
\begin{equation*}
C_{[-1,1]}^{(\eta,\tau)}:= \Big\{f \in C(-1,1)|\lim_{|y| \rightarrow 1} (f \rho^{(\eta,\tau)})(y)=0 \Big\}
\end{equation*}
 is a linear space of real--valued continuous functions defined on the interval $(-1,1)$ equipped with the norm
\begin{equation*}
||f||_{\rho^{(\eta, \tau)}}:=||f\rho^{(\eta, \tau)}||:=\max_{y \in [1,-1]}\Big\{\big|(f\rho^{(\eta, \tau)})(y)\big|\Big\},
\end{equation*}
where
\begin{equation*}
\rho^{(\eta,\tau)}(y):=(1-y)^\eta (1+y)^\tau, (y \in [-1,1])
\end{equation*}
 is the Jacobi weight with parameters $\eta,\tau \geq 0.$

The space $L_{[-1,1]}^{p,(\gamma,\delta)}$ with
the weight function $\rho^{(\gamma,\delta)}(y):=(1-y)^\gamma(1+y)^\delta ,\;\;\gamma,\delta>-1,$ is the class of all measurable functions $f$
on the segment $[-1,1],$ such that
\begin{equation*}
\int_{-1}^1 |f(y)\rho^{(\gamma,\delta)}(y)|^p dy < \infty.
\end{equation*}
This space is equipped with the norm
\begin{equation*}
||f||_{L_{[-1,1]}^{p,(\gamma,\delta)}}=\Big\{\int_{-1}^1 |f(y) \rho^{(\gamma,\delta)}(y)|^p dy\Big\}^{\frac{1}{p}},\; for \; p \geq 1.
\end{equation*}

In random series (\ref{1.11}), the random variables $A_n(\omega)$ are the Fourier--Jacobi coefficients of the symmetric stable process $X(t,\omega)$ of index $\alpha \in [1,2].$
We know that, the stochastic integral
\begin{equation}\label{0.1.1}
\int\limits_a^b f(t)dX(t,\omega)
\end{equation}
is defined in the sense of probability and is a random variable, if $f$ is a continuous function in $[a,b]$ and $X(t,\omega),$ for $t\in \mathbb{R},$ is a continuous stochastic process with independent increments \cite{L}.
Moreover, the stochastic integral
(\ref{0.1.1}) is defined in the sense of probability, for $f \in L_{[a,b]}^p, \; p \geq \alpha \geq 1$ and $X(t,\omega)$ is a symmetric stable process of index $\alpha \in[1,2]$ \cite{NPM}.
Accordingly, if $f\in L_{[-1,1]}^{p,(\eta,\tau)},\;p \geq 1,\eta,\tau >-1$ i.e. $f\rho^{(\eta,\tau)} \in L_{[-1,1]}^p,$ then the stochastic integral
\begin{equation*}
\int\limits_{-1}^1 f(t)\rho^{(\eta,\tau)}(t)dX(t,\omega)
\end{equation*}
 exists in probability, for $p \geq \alpha \geq 1 .$
A sequence of random variables $X_n$ is said to converge in probability to a random variable $X,$ if
$\lim\limits_{n\rightarrow \infty} P\Big(|X_n - X|>\epsilon\Big) = 0,$ for $\epsilon>0.$
In particular, if $f(t)$ is the orthonormal Jacobi polynomial $p_n^{(\gamma,\delta)}(t),\;\gamma,\delta>-1,$ then $p_n^{(\gamma,\delta)}(t)\rho^{(\eta,\tau)}(t)$ is continuous in $[-1,1],$ for all $\eta,\tau \geq 0$ and is in both the spaces $C_{[-1,1]}^{(\eta,\tau)}$
and $L_{[-1,1]}^{p,(\eta,\tau)}.$ Hence the integrals
\begin{equation}\label{1.13.0}
A_n(\omega):=\int_{-1}^1 p_n^{(\gamma,\delta)}(t)\rho^{(\eta,\tau)}(t)dX(t,\omega)
\end{equation}
exist and are random variables, which are not independent, for each $n=1,2,\dots$ .
The series (\ref{1.11}) is called the random Fourier--Jacobi series of function $f.$
The summability of random Fourier--Jacobi series (\ref{1.11}) of functions $f$ in the space $C_{[-1,1]}^{(\eta,\tau)}$ and
$ L_{[-1,1]}^{p,(\gamma,\delta)},\;p \geq 1$ are studied in this article.

The summability in the space $C_{[-1,1]}^{(\eta,\tau)}$ is discussed by using the summation matrix $\Theta$ which
 is a lower triangular infinite matrix (see equation (\ref{1.1.1})). We find different summability methods for different values of the entries of the matrix $\Theta.$ The various summability of the random series (\ref{1.11}) depend on different conditions on
 the summation matrix $\Theta.$
The N{\"o}rlund summability and some other variants of it are established,
for the random Fourier--Jacobi series of functions in the space $L_{[-1,1]}^{p,(\gamma,\delta)},\;p=1.$ Each of these summability methods has conditions on the parameters $\gamma,\delta$ that differ from the conditions on these parameters in the case of the Fourier--Jacobi series of functions $f \in L_{[-1,1]}^{1,(\gamma,\delta)}.$ 
If $0 \leq \gamma <1/2$ and $\delta \geq 0,$ the random Fourier--Jacobi series is N{\"o}rlund summable. In this case, however the Fourier--Jacobi series is summable, for $-1/2 \leq \gamma <1/2$ and $\delta >-1/2.$
The generalized N{\"o}rlund summability of random Fourier--Jacobi series holds for $\gamma \geq 0,$ $\delta-\gamma >1$ and $\delta +\gamma >0,$ but it is observed for the Fourier--Jacobi series if $\gamma \geq 0, \delta -\gamma >1$ and $\delta +\gamma \geq -1.$
Furthermore, the random Fourier--Jacobi series is lower triangular summable for $0 \geq \gamma \geq 1/2$ and $\delta \geq 0.$ The Fourier--Jacobi series, on the other hand, is lower triangular summable if $-1/2 \geq \gamma \geq 1/2$ and $\delta >-1/2.$

This article is organized as follows.
The preliminaries and some results relevant to this work are stated in Section 2.
The $\Theta$--summability in general of the random Fourier--Jacobi series (\ref{1.11}) associated with symmetric stable process of index $\alpha \in [1,2]$ under some conditions on $\gamma,\delta,\eta$ and $\tau$ is discussed in Section 3.
Further, the other summability of the random Fourier--Jacobi series (\ref{1.11}), such as Ces{\'a}ro, Riesz, Rogosinski, etc., for different set of values of entries in the matrix $\Theta$ are established.
 Section 4 investigates the summability of random Fourier--Jacobi series (\ref{1.11}) of functions $f \in L_{[-1,1]}^{1,(\gamma,\delta)}.$
We prove the N{\"o}rlund summability, generalized N{\"o}rlund summability, and lower triangular matrix summability of the random series (\ref{1.11}) associated with the symmetric stable process of index $\alpha=1,$ for the parameter $\gamma,\delta \geq 0$ at a particular point only.
\section{Preliminaries and results on summability of Fourier--Jacobi series}
 Some definitions and results from the work of Chripk{\'o}, Raghuvanshi, Thorpe, and Dhakal \cite{C,dha,rag,to} are presented here which will be of use in proving our results in next sections.

Chripk{\'o} \cite{C} discussed the $\Theta$--summability of Fourier--Jacobi series
\begin{equation}\label{1.1}
\sum_{n=0}^{\infty} a_n p_n^{(\gamma,\delta)}(y)
\end{equation}
in the weighted space $C_{[-1,1]}^{(\eta, \tau)}$ of
continuous functions.
He considered the summation matrix $\Theta$ as
\begin{equation}\label{1.1.1}
\Theta :=
\begin{bmatrix}
\theta_{0,1} & & &   \\
\theta_{0,2} & \theta_{1,2} & & \\
\theta_{0,3} & \theta_{1,3} & \theta_{2,3}\\
\vdots & \vdots &  \vdots\\
 \vdots &  \vdots &  \vdots \ddots\\
\end{bmatrix},
\end{equation}
where $\theta_{k,n}$'s are the real entries.
The $\Theta$--sum of a Fourier--Jacobi series (\ref{1.1}) is defined as
\begin{equation}\label{1.8}
\mathbf{s}_n^{\Theta,(\gamma,\delta)} (f,y):=\sum_{k=0}^{n-1} \theta_{k,n} a_k p_k^{(\gamma, \delta)}(y),\; (y \in [-1,1],\;n \in \mathbb{N}),
\end{equation}
where $a_k$ are the Fourier--Jacobi coefficients of a function $f \in C_{[-1,1]}^{(\eta,\tau)}.$\\
Chripk{\'o} considered some choices of the summation matrix $\Theta,$ which are stated below.\\
\textbf{Case-1}\\
If $\theta_{k,n}:=1\;(k=0,1,2\dots,n-1,n \in \mathbb{N}),$ then the $n$th Fourier--Jacobi series (\ref{1.8}) of $f$ is equal to the $n$th partial sum of the Fourier--Jacobi series (\ref{1.1}).\\
\textbf{Case-2}\\
Let $$\theta_{k,n}:= \frac{B_{n-k-1}^{(\phi)}}{B_{n-1}^{(\phi)}}\;(\phi \geq 0,\;k=0,\dots,n-1,\; n\in \mathbb{N}),$$
where $B_0^{(\phi)}:=1,\;B_m^{(\phi)}:=\begin{pmatrix}
m+\phi\\
m
\end{pmatrix}=\displaystyle{\frac{(\phi+1)\dots(\phi+m)}{m!}},\;(m \in \mathbb{N}).$
Then $\Theta:= (\theta_{k,n})$ is the Ces{\'a}ro summation matrix and the Ces{\'a}ro means of the Fourier--Jacobi series (\ref{1.1}) is defined as
\begin{equation}\label{1.9}
\sigma_n^{(\phi),\gamma,\delta}(f,y):=\sum_{k=0}^{n-1} \frac{B_{n-k-1}^{(\phi)}}{B_{n-1}^{(\phi)}} a_k p_k^{(\gamma,\delta)}(y),\; y \in [-1,1],\;n \in \mathbb{N}
\end{equation}
\begin{equation*}
= \int_{-1}^1 \mathbf{F}_n^{(\phi),\gamma,\delta}(y,t)\rho^{(\gamma,\delta)}(t)dt,
\end{equation*}
where
\begin{equation*}
 \mathbf{F}_n^{(\phi),\gamma,\delta}(y,t):=\sum_{k=0}^{n-1}\frac{B_{n-k-1}^{(\phi)}}{B_{n-1}^{(\phi)}}p_k^{(\gamma,\delta)}(t)p_k^{(\gamma,\delta)}(y).
\end{equation*}
If $\phi=0,$ then it is the same as in Case-1. For $\phi \geq 1,$ the sum (\ref{1.9}) is known as  Ces{\`a}ro $(C,\phi)$ means of the Fourier--Jacobi series (\ref{1.1}). In particular, for $\phi=1,$ we obtain the Fejer summation of the Fourier--Jacobi series (\ref{1.1}). \\
\textbf{Case-3}\\
Let $\theta_{k,n}:= \Big[1-\Big(\displaystyle{\frac{k}{n}}\Big)^{\nu}\Big]^\mu,\; k=0,1,\dots,n-1,n \in \mathbb{N},$
where $\nu,\; \mu \geq 0$ are constant real numbers.
Then
\begin{equation*}
\mathbf{s}_n^{(\gamma,\delta)_{\nu,\mu}}(f,y):=\sum_{k=0}^{n-1}\Big[1-\Big({\frac{k}{n}}\Big)^{\nu}\Big]^\mu a_k p_k^{(\gamma,\delta)}(y), \;(y \in [-1,1])
\end{equation*}
is the Riesz $(R,\nu,\mu)$ sum of the Fourier--Jacobi series (\ref{1.1}).\\
\textbf{Case-4}\\
Let \begin{center}
\(\theta_{k,n}:= \begin{pmatrix}
1,\;\; if \; 0\leq \displaystyle{\frac{k}{n}}\leq s \\
\displaystyle{\frac{\frac{k}{n}-1}{s-1}}, \;\; if \; s < \frac{k}{n} \leq 1\\
\end{pmatrix}
\),
\end{center}
where $k \in 0,1,2,\dots,n-1,\; n \in \mathbb{N}$ and $s \in (0,1).$
Then
 \begin{equation*}
\mathbf{s}_n^{(\gamma,\delta),s}(f,y):=\sum_{k=0}^{n-1} \theta_{k,n} a_k p_k^{(\gamma,\delta)}(y)
\end{equation*}
is the de la Vall{\'e}e Poussin means  of the Fourier--Jacobi series (\ref{1.1}).\\
\textbf{Case-5}\\
Let $\theta_{k,n}:=cos\displaystyle{\frac{\pi k}{2n}},$ for $k=0,1,\dots ,n-1, \; n \in \mathbb{N},$ then
\begin{equation*}
\mathbf{s}_n^{(\gamma,\delta),R}(f,y):=\sum_{k=0}^{n-1} \theta_{k,n} a_k p_k^{(\gamma,\delta)}(y),\; y \in[-1,1]
\end{equation*}
is known as the Rogosinski means of the Fourier--Jacobi series (\ref{1.1}).

Choudhary \cite{RSC} discussed the regular N{\"o}rlund summability (also known as N{\"o}rlund summability) of Fourier--Jacobi series (\ref{1.1}).
The regular N{\"o}rlund sum of a series is defined below:\\
Let $\sum\limits_{n=0}^\infty a_n$ be a given infinite series with the sequence of partial sums $\{s_n\}.$ Let $p_n$ be a sequence of constants,
real or complex and let $\mathrm{P}_n=p_0+p_1+\dots+p_n.$
Define the sequence to sequence transformation $h_n$ as
\begin{equation*}
  h_n:=\frac{1}{\mathrm{P}_n} \sum_{v=0}^n p_{n-v}s_v=\frac{1}{\mathrm{P}_n} \sum_{v=0}^n p_{v}s_{n-v},\;\mathrm{P}_n \neq 0
\end{equation*}
The sequence $\{h_n\}$ is said to be N{\"o}rlund means of $\{s_n\},$ generated by the constants $\{p_n\}.$ The series
$\sum\limits_{n=0}^\infty a_n$ is N{\"o}rlund summable to the sum $s$ if $\lim\limits_{n \rightarrow \infty} h_n$ exists and equal to $s.$
\\
The N{\"o}rlund sum of Fourier--Jacobi series at the point $y=1$ is
\begin{equation*}
\mathbf{h}_n^{(\gamma,\delta)}(f,y):= \frac{1}{\mathrm{P}_n} \sum_{k=0}^n p_{n-k}\mathbf{s}_k^{(\gamma,\delta)}(f,1),
\end{equation*}
where $\mathbf{s}_k^{(\gamma,\delta)}(f,1)$ is the $n$th partial sum of the Fourier--Jacobi series (\ref{1.1}) at the point $y=1.$
We use the symbol $(\mathbf{N},p_n)$ for it.
The following theorem is established by Choudhary.
\begin{theorem}\cite{RSC}\label{t.1}
Let $(\mathbf{N},p_n)$ be a N{\"o}rlund method defined by a real non--negative monotonic non--increasing sequence of coefficients $\{p_n\}$ such that
\begin{equation}\label{5.13.1}
\mathrm{P}_n \rightarrow \infty\; as \ n \rightarrow \infty.
\end{equation}
 If
\begin{equation}\label{5.12}
\Psi(t)=O \Big[\displaystyle{\frac{p(1/t)}{\mathrm{P}(1/t)}t^{(2\gamma+1)}} \Big],\; as  \; t \rightarrow 0,
\end{equation}
and
\begin{equation}\label{5.12.1}
\sum\limits_{n=0}^\infty \displaystyle{\frac{n^{\gamma+1/2}}{\mathrm{P}_n}}< \infty,
\end{equation}
 then
 the series (\ref{1.1}) of some function $f \in L_{[-1,1]}^{1,(\gamma,\delta)}$
 is $(\mathbf{N},p_n)$ summable at the point $y=1$ to the sum $A$ provided
 $-1/2 \leq \gamma < 1/2,\; \delta>-1/2$ and
 the antipole condition
\begin{equation}\label{5.13}
\int_{-1}^b {(1+x)}^{\delta/2-3/4}|f(x)|dx <\infty,
\end{equation}
is satisfied, for fixed $b.$
\end{theorem}

Raghuvanshi \cite{rag} discussed the generalized N{\"o}rlund summability of Fourier--Jacobi series (\ref{1.1}).
The generalized N{\"o}rlund summability of a series is defined as below:\\
Let $\{s_n\}$ be the sequence of partial sums of an infinite series $\sum\limits_{n=0}^\infty \mathbf{a}_n.$ Let $\{r_n\}$ and $\{q_n\}$ be any two sequences of positive real constants
with $R_n$ and $Q_n$ as their $n$th partial sums respectively,
and
let
\begin{equation*}
(q*r)_n:=\sum_{k=0}^n q_{n-k}r_k=\sum_{k=0}^n q_k r_{n-k}
\end{equation*}
 tends to infinity as $n \rightarrow \infty.$
As in \cite{bon}, define the sequence to sequence transformation as
\begin{eqnarray*}
t_n^{q,r}:=\frac{1}{(q\ast r)_n} \sum_{k=0}^n q_{n-k} r_k s_k.
\end{eqnarray*}
If
$t_n^{q,r} \rightarrow s \; as \; n\rightarrow \infty,$
then the infinite series $\sum\limits_{n=0}^\infty \mathbf{a}_n$ is said to be generalized N{\"o}rlund summable to $s.$ We use the notation $|\mathbf{N},q_n,r_n|$ for generalized N{\"o}rlund summation.
The $|\mathbf{N},q_n,r_n|$ sum of the Fourier--Jacobi series (\ref{1.1}) at the point $y=1$ is defined as
\begin{equation*}
\mathbf{u}_n^{(\gamma,\delta)}(f,1):=\frac{1}{(q\ast r)_n} \sum_{k=0}^n q_{k} r_{n-k} \mathbf{s}_k^{(\gamma,\delta)}(f,1) ,
\end{equation*}
where $\mathbf{s}_k^{(\gamma,\delta)}(f,1)$ is the $n$th partial sum of the Fourier--Jacobi series (\ref{1.1}) at the point $y=1.$
The following result is established by  Raghuvanshi.
\begin{theorem}\cite{rag}\label{5t.1}
Let $|\mathbf{N},q_n,r_n|$ be the generalized summability method defined by a non--negative real constants sequences $\{q_n\},\;\{r_n\}$
 and let $\gamma>-1/2,\; \delta-\gamma >1,\; \delta+\gamma \geq -1 $ such that
\begin{equation}\label{1.15}
\sum_{k=2}^n \frac{(q \ast r)_k}{k^{\gamma+(1/2)} \log k}= O\Big(\frac{{(q \ast r)_n}}{n^{\gamma+(1/2)}}  \Big)\; as \; n \rightarrow \infty.
\end{equation}
Also suppose that
\begin{equation}\label{1.16}
\int_{1-t}^t|f(u)-A|du=O \Big(\frac{t}{\log(1/t)} \Big) \; as \; t \rightarrow 0,
\end{equation}
and the antipole condition
\begin{equation}\label{1.17}
\int_{-1}^b (1+y)^{(\delta-\gamma-1)/2} |f(y)|dy< \infty, \; b \; fixed,
\end{equation}
are satisfied, then the series (\ref{1.1}) is $|\mathbf{N},q_n,r_n|$ summable to the sum $A$ at the point $y=1.$
\end{theorem}
Dhakal \cite{dha} studied the summability of Jacobi series in lower triangular matrix method. The lower triangular matrix method of a series is defined as below:\\
Let $M=(\mathbf{b}_{n,k})$ be an infinite lower triangular matrix satisfying the Silverman-T{\"o}eplitz conditions \cite{to} of regularity i.e.
\begin{equation*}
\sum_{k=0}^n \mathbf{b}_{n,k} \rightarrow 1\; as \; n \rightarrow \infty,
\end{equation*}
where $\mathbf{b}_{n,k}=0,$ for $k>n$ and $\sum\limits_{k=0}^n |\mathbf{b}_{n,k}| \leq M,$ for $M$ a finite $+$ve constant.\\
Consider $s_n$ be the $n$th partial sum of an infinite series $\sum\limits_{n=0}^\infty \mathbf{a}_n.$
Let $\{t_n\}$ be the lower triangular matrix sums of $\{s_n\},$ generated by the sequence of coefficient $(\mathbf{b}_{n,k})$ i.e.
the sequence-to-sequence transformation is
\begin{equation*}
t_n:=\sum_{k=0}^n \mathbf{b}_{n,k} s_k.
\end{equation*}
If $t_n \rightarrow s \; as \; n \rightarrow \infty,$ then the series $\sum\limits_{n=0}^\infty \mathbf{a}_n$ is said to be summable by the lower triangular matrix method to $s.$ We use the notation $(\mathbf{T})$ for lower triangular matrix summation. \\
The $(\textbf{T})$ summation of Jacobi series (\ref{1.1}) at the point $y=1$ is
\begin{equation*}
\mathbf{t}_n^{(\gamma,\delta)}(f,1)=\sum_{k=0}^n \mathbf{b}_{n,k} \mathbf{s}_k^{(\gamma,\delta)}(f,1),
\end{equation*}
where $\mathbf{s}_k^{(\gamma,\delta)}(f,1)$ is the $n$th partial sum of the Fourier--Jacobi series (\ref{1.1}) at the point one.
The following result is established by Dhakal on $(\textbf{T})$ summability of Fourier--Jacobi series (\ref{1.1}) of the function $f \in L_{[-1,1]}^{1,(\gamma,\delta)}.$
\begin{theorem}\cite{dha}\label{5t.2}
Let $M=(\mathbf{b}_{n,k})$ be an infinite lower triangular regular matrix such that the element $(\mathbf{b}_{n,k})$ is positive, monotonic increasing in $k,\;0 \leq k \leq n,$
\begin{equation}\label{1.18}
D_{n,\tau}=\sum_{k=n-\tau}^n \mathbf{b}_{n,k},\; D_{n,n}=1,\; for \; all \; n,
\end{equation}
and
\begin{equation}\label{1.19}
n^{\gamma+1/2}D_{n,[1/\beta]}=o(1),0< \beta <\pi \;as \; n \rightarrow \infty,
\end{equation}
where $\tau=Integral\; part \; of \; \frac{1}{\phi}=[\frac{1}{\phi}]$ and $0 \leq \phi \leq \pi.$
For $-1/2 \leq \gamma \leq 1/2,\; \delta >-1/2,$ if
\begin{equation}\label{1.20}
\int_{1-t}^1|f(u)-A|du=o\Big(\frac{t}{\xi(1/t) \log (1/t)}\Big)\; as \; t\rightarrow 0,
\end{equation}
then the Jacobi series (\ref{1.1}) is $(\mathbf{T})$ summable to the sum $A$ at $y=1$ provided $\xi(t)$ is positive monotonic non-decreasing function of $t$ such that
\begin{equation}\label{1.21}
\sum_{k=a}^n \frac{D_{n,k}}{k^{(2\gamma+3)/2}\xi(k) \log k}=O\Big(\frac{1}{n^{(2\gamma+1)/2}}\Big),
\end{equation}
and satisfy the antipole condition
\begin{equation}\label{1.22}
\int_0^{1/n} t^{\delta-1/2}|f(-cost)-A| dt=o(1)\; as \; n \rightarrow \infty.
\end{equation}
\end{theorem}
\section{Summability of random Fourier--Jacobi series in the space $C_{[-1,1]}^{(\eta,\tau)}$ }
Consider the random Fourier--Jacobi series
\begin{equation}\label{1.11.0}
\sum_{n=0}^\infty a_nA_n(\omega)p_n^{(\gamma,\delta)}(y),
\end{equation}
where the scalars $a_n$ are the Fourier--Jacobi coefficients of a function $f \in C_{[-1,1]}^{(\eta,\tau)},\; \eta,\tau \geq 0$ defined as in (\ref{1.2})
and the random variables $A_n(\omega)$ are defined as in (\ref{1.13.0}).
The $\Theta$--sum of the random Fourier--Jacobi series (\ref{1.11.0}) is
\begin{equation}\label{1.8.1}
\mathbf{S}_n^{\Theta,(\gamma,\delta)}(f,y,\omega):=\sum_{k=0}^{n-1} \theta_{k,n}a_k A_k(\omega)p_k^{(\gamma,\delta)}(y).
\end{equation}
We will use the following notations on $\Theta$--summation matrix same as in
 Chripk{\'o} \cite{C}.
\begin{eqnarray*}
&&T_1 : \lim_{n \rightarrow \infty}(1- \theta_{k,n})=0,\; for \; all \; fixed \; k=0,1,2,\dots.\\
&&T_2 : \theta_{n-1,n}=O
\Big(\frac{1}{n}\Big),\;( n\in \mathbb{N}).\\
&&T_3: \bigtriangleup^2 \theta_{k-1,n}=
O\Big(\frac{1}{n^2}\Big),\; (k=1,2,\dots, n-1, n\in \mathbb{N}).\\
&&T_4:\bigtriangleup^2 \theta_{k-1,n} \; (k=1,2,\dots, n-1, n\in \mathbb{N})\; is \; of \; constant \; sign.\\
&&T_5 : sgn \bigtriangleup^2 \theta_{k-1,n}= sgn\; \theta_{n-1,n} (k=1,2,\dots, n-1, n\in \mathbb{N}),
\end{eqnarray*}
where $\bigtriangleup^2\theta_{k,n}:=\bigtriangleup\theta_{k+1,n}-\bigtriangleup\theta_{k,n},
\quad \bigtriangleup \theta_{k,n}:=\theta_{k+1,n}-\theta_{k,n}, \; (\theta_{n,n}:=0).$\\
Furthermore, the following notations will be used.\\
$(A)$ if $(T_1),$ $(T_2)$ and $(T_3)$ hold,\\
$(B)$ if $(T_1),$ $(T_2)$ and $(T_4)$ hold,\\
$(C)$ if $(T_1),$ $(T_5)$ hold.\\
The convergence of $\Theta$--sum (\ref{1.8.1}) of the random Fourier--Jacobi series (\ref{1.11.0}) is proved in Theorem \ref{T5} by imposing some restrictions on $\gamma,\delta,\eta,\tau$ and the entries of summation matrix $\Theta.$
To prove this theorem, we need the following result.
\begin{lemma} \label{l1} \cite{NPM}
Let $X(t,\omega)$ be a symmetric stable process of index  $\alpha,\;1 \leq \alpha \leq 2$ and $f(t)$ be any function in $L^p_{[a,b]},\;p\geq 1$, then for all $\epsilon >0,$
$$P\Bigg(\Bigg|\int_a^b f(t)dX(t,\omega)\Bigg|> \epsilon\Bigg)
\leq
\frac{C2^{\alpha+1}}{(\alpha+1)\epsilon'^{\alpha}}
\int_a^b|f(t)|^\alpha dt,
$$
where $C$ is a positive constant and $\epsilon'<\epsilon.$
\end{lemma}
\begin{theorem}\label{T5}
The $\Theta$--sum (\ref{1.8.1}) of random Fourier--Jacobi series (\ref{1.11.0}) converges in probability to the stochastic integral
\begin{equation}\label{1.14}
\int_{-1}^1 f(y,t) \rho^{(\eta,\tau)}(t) dX(t,\omega),
\end{equation}
if the parameters $\gamma,\delta \geq -1/2,\eta,\tau \geq 0$ satisfy the following conditions
\begin{equation}\label{1.10}
\frac{\gamma}{2}-\frac{1}{4}< \eta < \frac{\gamma}{2}+\frac{3}{4},\; and \;\frac{\delta}{2}-\frac{1}{4}< \tau < \frac{\delta}{2}+\frac{3}{4},
\end{equation} with the entries of $\Theta$--matrix satisfying either (A), (B) or (C).
\end{theorem}
\begin{proof}
The integral form of the $n$th partial $\Theta$--sum of the random Fourier--Jacobi series (\ref{1.11.0}) involving Jacobi polynomials is
\begin{eqnarray*}
\mathbf{S}_n^{\Theta,(\gamma,\delta)}(f,y,\omega)&:=& \sum_{k=0}^{n-1} \theta_{k,n} a_k A_k(\omega) p_k^{(\gamma,\delta)}(y)\\
&=& \sum_{k=0}^{n-1} \theta_{k,n} a_k p_k^{(\gamma,\delta)}(y) \int_{-1}^1 p_k^{(\gamma,\delta)}(t)\rho^{(\eta,\tau)}(t)dX(t,\omega)\\
&=& \int_{-1}^1 \sum_{k=0}^{n-1} \theta_{k,n} a_k p_k^{(\gamma,\delta)}(y)p_k^{(\gamma,\delta)}(t)\rho^{(\eta,\tau)}(t)dX(t,\omega)\\
&=&\int_{-1}^1 \mathbf{s}_n^{\Theta,(\gamma,\delta)}(f,y,t)\rho^{(\eta,\tau)}(t)dX(t,\omega),
\end{eqnarray*}
where
\begin{equation*}
\mathbf{s}_n^{\Theta,(\gamma,\delta)}(f,y,t):=\sum_{k=0}^{n-1} \theta_{k,n} a_k p_k^{(\gamma,\delta)}(t)p_k^{(\gamma,\delta)}(y).
\end{equation*}
Now
 \begin{eqnarray*}
 &&\mathbf{S}_n^{\Theta,(\gamma,\delta)}(f,y,\omega)-\mathbf{S}_m^{\Theta,(\gamma,\delta)}(f,y,\omega)\\
&=&\int_{-1}^1\Big(\mathbf{s}_n^{(\Theta,(\gamma,\delta)}(f,y,t)-\mathbf{s}_m^{\Theta,(\gamma,\delta)}(f,y,t)\Big)\rho^{(\eta, \tau)}(t)dX(t,\omega).
 \end{eqnarray*}
By the inequality in Lemma \ref{l1}, we obtain
\begin{eqnarray*}
  &&P\Big(\Big|\mathbf{S}_n^{\Theta,(\gamma,\delta)}(f,y,\omega)-\mathbf{S}_m^{\Theta,(\gamma,\delta)}(f,y,\omega)\Big|>\epsilon\Big)\\
  &=&P\Big(\Big|\Big(\mathbf{s}_n^{\Theta,(\gamma,\delta)}(f,y,t)-\mathbf{s}_m^{\Theta,(\gamma,\delta)}(f,y,t)\Big)\rho^{(\eta, \tau)}(t)dX(t,\omega)\Big|>\epsilon\Big)\\
  & \leq & \frac{C2^{\alpha+1}}{(\alpha+1)\epsilon'^\alpha}
\int_{-1}^1 \Big|\Big(\mathbf{s}_n^{\Theta,(\gamma,\delta)}(f,y,t)-\mathbf{s}_m^{\Theta,(\gamma,\delta)}(f,y,t)\Big)\rho^{(\eta,\tau)}(t)\Big|^\alpha dt,\; for \; \epsilon'<\epsilon.
\end{eqnarray*}
Invoking Chripk{\'o} result (see \cite[Theorem~3.2]{C}) on the convergence of Fourier--Jacobi series (\ref{1.1}),
we know that $\mathbf{s}_n^{\Theta,(\gamma,\delta)}$ converges uniformly in $C_{[-1,1]}^{(\eta,\tau)},$ if $a_n$ are the Fourier--Jacobi coefficients of a function $f$ in $C_{[-1,1]}^{(\eta,\tau)}$ and
 $\gamma,\delta,\eta,\tau$ satisfy the conditions in equation (\ref{1.10}).
 Hence $\mathbf{s}_n^{\Theta,(\gamma,\delta)}$  is a Cauchy sequence.\\
This yields, for $\alpha \in [1,2]$ and $f \in C_{[-1,1]}^{(\eta,\tau)},$
\begin{equation*}
\lim_{n,m \rightarrow \infty}\int_{-1}^1\Big|\Big(\mathbf{s}_n^{\Theta,(\gamma,\delta)}(f,y,t)-\mathbf{s}_m^{\Theta,(\gamma,\delta)}(f,y,t)\Big)\rho^{(\eta, \tau)}(t)\Big|^\alpha dt=0.
\end{equation*}
This result holds uniformly for all $y,t$ in the range $[-1,1],$ implying that $\mathbf{S}_n^{\Theta,(\gamma,\delta)}(f,y,\omega)$ is a Cauchy sequence in the sense that it converges in probability.
Hence, it will converge to a random variable.
Now we will show that the sequence $\mathbf{S}_n^{\Theta,(\gamma,\delta)}(f,y,\omega)$ converges to the stochastic integral $\int_{-1}^1 f(y,t)\rho^{(\eta,\tau)}(t)dX(t,\omega)$ in probability.
By Lemma \ref{l1},
\begin{eqnarray*}
&&P\Bigg(\Big|\mathbf{S}_n^{\Theta,(\gamma,\delta)}(f,y,\omega)-\int_{-1}^1f(y,t)\rho^{(\eta, \tau)}(t)dX(t,\omega)\Big|> \epsilon\Bigg)\\
&=&P\Bigg(\Big|\int_{-1}^1\Big(\mathbf{s}_n^{\Theta,(\gamma,\delta)}(f,y,t)-f(y,t)\Big)\rho^{(\eta, \tau)}(t) dX(t,\omega)\Big| > \epsilon\Bigg)\\
&\leq& \frac{C2^{\alpha+1}}{(\alpha+1)\epsilon'^\alpha}
\int_{-1}^1 \Big|\Big(\mathbf{s}_n^{\Theta,(\gamma,\delta)}(f,y,t)-f(y,t)\Big)\rho^{(\eta, \tau)}(t)\Big|^\alpha dt, \; for \;\epsilon'<\epsilon.
\end{eqnarray*}
 If $\gamma,\delta,\eta,\tau$ satisfy the conditions in (\ref{1.10}), then by Chripk{\'o} \cite[Theorem~3.2]{C},
\begin{equation*}
\lim_{n\rightarrow \infty}\int_{-1}^1\Big|\Big(\mathbf{s}^{\Theta,(\gamma,\delta)}_n(f,t)-f(t)\Big)\rho^{(\eta, \tau)}(t)\Big| dt=0.
\end{equation*}
Hence, for $\alpha \in [1,2],$
\begin{equation*}
\lim_{n\rightarrow \infty}\int_{-1}^1\Big|\Big(\mathbf{s}_n^{\Theta,(\gamma,\delta)}(f,y,t)-f(y,t)\Big)\rho^{(\eta, \tau)}(t)\Big|^\alpha dt=0.
\end{equation*}
This proves that the $\Theta$--sum (\ref{1.8.1}) of random Fourier--Jacobi series (\ref{1.11.0}) converges in probability to the stochastic integral (\ref{1.14}).
\end{proof}

The following theorem demonstrates that the random Fourier--Jacobi series (\ref{1.11.0}) is  Ces{\'a}ro $(C,\phi),\phi \geq 1$ summable, if $a_n$ are the Fourier--Jacobi coefficients of some function $f$ in the
space $C_{[-1,1]}^{(\eta,\tau)}$ under the same conditions on $\gamma,\delta,\; \eta$ and $\tau$ as stated in previous theorem.
\begin{theorem} \label{T6}
The random Fourier--Jacobi series (\ref{1.11.0}) is Ces{\`a}ro $(C,\phi), \; \phi \geq 1$ summable in probability to the stochastic integral
(\ref{1.14}), if  $\eta,\;\tau\geq 0,\gamma,\delta \geq -1/2$ and the entries of $\Theta$--matrix satisfy the conditions (\ref{1.10}) in Theorem \ref{T5}.
\end{theorem}
\begin{proof}
The $n$th $(C,\phi),\phi \geq 1$ means of random Fourier--Jacobi series (\ref{1.11.0}) is
\begin{equation}\label{1.23}
\mathbf{Q}_n^{(\phi),\gamma,\delta}(f,y,\omega):=\sum_{k=0}^{n-1} \frac{B_{n-k-1}^{(\phi)}}{B_{n-1}^{(\phi)}} \mathbf{S}_k^{(\gamma,\delta)}(f,y,\omega),
\end{equation}
where $B_m^n=\displaystyle{\frac{\Gamma(m+n+1)}{\Gamma(m+1)\Gamma(n+1)}}$ and $\mathbf{S}_k^{(\gamma,\delta)}(f,y,\omega)$ is the $n$th partial sum of random Fourier--Jacobi series (\ref{1.11.0}).
The $(C,\phi)$ means (\ref{1.23}) can be expressed in the integral form as
\begin{eqnarray*}
\mathbf{Q}_n^{(\phi),\gamma,\delta}(f,y,\omega)&:=&\sum_{k=0}^{n-1} \frac{B_{n-k-1}^{(\phi)}}{B_{n-1}^{(\phi)}} \mathbf{S}_k^{(\gamma,\delta)}(f,y,\omega)\\
&=&\int\limits_{-1}^1\sigma_n^{(\phi),\gamma,\delta}(f,y,t)\rho^{(\eta, \tau)}(t)dX(t,\omega),
\end{eqnarray*}
where
\begin{equation*}
\sigma_n^{(\phi),\gamma,\delta}(f,y,t):=\sum_{k=0}^{n-1} \frac{B_{n-k-1}^{(\phi)}}{B_{n-1}^{(\phi)}} \mathbf{s}_n^{(\gamma,\delta)}(f,y,t),
\end{equation*}
and $\mathbf{s}_n^{(\gamma,\delta)}(f,y)$ is the $n$th partial sum of Fourier--Jacobi series (\ref{1.1}).\\
Now
 \begin{eqnarray*}
&& \mathbf{Q}_n^{(\phi),\gamma,\delta}(f,y,\omega)-\mathbf{Q}_m^{(\phi),\gamma,\delta}(f,y,\omega)\\
&=&\int_{-1}^1\Big(\sigma_n^{(\phi),\gamma,\delta}(f,y,t)-\sigma_m^{(\phi),\gamma,\delta}(f,y,t)\Big)\rho^{(\eta, \tau)}(t)dX(t,\omega).
 \end{eqnarray*}
By the inequality in Lemma \ref{l1}, we obtain $\mathbf{Q}_n^{(\phi),\gamma,\delta}(f,y,\omega)$ is a Cauchy sequence.
With the help of Chripk{\'o} result (see \cite[Corollary~3.4]{C}) on $(C,\phi),\phi \geq 1$ summation and
the Lemma \ref{l1}, the Ces{\`a}ro means of the random series (\ref{1.11.0}) converges in probability to the stochastic integral $\int_{-1}^1 f(y,t)\rho^{(\eta,\tau)}(t)dX(t,\omega)$ in probability.
\end{proof}

The following theorem states the summability of random Fourier--Jacobi series (\ref{1.11.0}) in some other means under the same conditions (\ref{1.10}) in Theorem \ref{T5} by using the Lemma \ref{l1},
for $f$ in $C_{[-1,1]}^{(\eta,\tau)},\eta,\tau \geq 0.$
\begin{theorem}\label{T7}
The random Fourier--Jacobi series (\ref{1.11.0}) is
\\(a) \textit{Riesz} $(R,\mu,\nu),$ for $\mu,\nu \geq 1,$
\\(b)  \textit{de la Vall{\'e}e Poussin}, for every $s \in (0,1),$
\\(c)  \textit{Rogosinski}
summable in probability to the stochastic integral (\ref{1.14}),
if $\eta,\tau\geq 0,\gamma,\delta \geq -1/2$ satisfy the conditions in Theorem \ref{T5}.
\end{theorem}
\section{ Summability of random Fourier--Jacobi series in $L_{[-1,1]}^{1,(\gamma,\delta)}$ space}
Let $f \in L_{[-1,1]}^{1,(\gamma,\delta)},\; \gamma,\delta \geq 0$ and $a_n$ be the Fourier--Jacobi coefficients of a function $f.$
Consider the random Fourier--Jacobi series as
\begin{equation}\label{1.11.2}
\sum_{n=0}^\infty a_nA_n(\omega) p_n^{(\gamma,\delta)}(y),
\end{equation}
where $a_n$ are Fourier--Jacobi coefficients of a function $f \in L_{[-1,1]}^{1,(\gamma,\delta)}$ and
  $A_n(\omega)$ defined as
 \begin{equation}\label{1.13.1}
\int_{-1}^1 p_n^{(\gamma,\delta)}(t)\rho^{(\gamma,\delta)}(t) dX(t,\omega)
 \end{equation}
are the Fourier--Jacobi coefficients of symmetric stable process $X(t,\omega)$ of index $\alpha=1.$
It is shown in the following  theorem that the N{\"o}rlund $(\mathbf{N},p_n)$ sum of
 the random Fourier--Jacobi series (\ref{1.11.2}) converges
 in probability to the stochastic integral
\begin{equation}\label{5.3}
\int_{-1}^1 A \rho^{(\gamma,\delta)}(t)dX(t,\omega),
\end{equation}
where $A$ is a positive fixed constant.
\begin{theorem}\label{5T.1}
If $p_n$ is the non--negative non--increasing sequences satisfy the conditions (\ref{5.13.1}), (\ref{5.12}), (\ref{5.12.1}) along with the antipole condition (\ref{5.13}) stated in the Theorem \ref{t.1} and $0 \leq \gamma < 1/2,\; \delta \geq 0,$
then the random Fourier--Jacobi series (\ref{1.11.2}) is $(\mathbf{N},p_n)$ summable in probability to the stochastic integral
(\ref{5.3}) at the point one.
\end{theorem}
\begin{proof}
Let \begin{equation}\label{0}
\mathbf{S}_n^{(\gamma,\delta)}(f,y,\omega):=\sum_{k=0}^n a_k A_k(\omega)p_k^{(\gamma,\delta)}(y)
\end{equation}
be the $n$th partial sum of the random Fourier--Jacobi series (\ref{1.11.2}).
For $y=1,$ the integral form of $\mathbf{S}_n^{(\gamma,\delta)}(f,1,\omega)$ is
\begin{equation*}
\mathbf{S}_n^{(\gamma,\delta)}(f,1,\omega):=\sum_{k=0}^n a_kA_k(\omega)p_k^{(\gamma,\delta)}(1)
\end{equation*}
\begin{equation}\label{0.1}
=\int_{-1}^1 \mathbf{s}_n^{(\gamma,\delta)}(1,t) \rho^{(\gamma,\delta)}(t)dX(t,\omega),
\end{equation}
where $$\mathbf{s}_n^{(\gamma,\delta)}(f,1,t):=\sum\limits_{k=0}^n a_k p_k^{(\gamma,\delta)}(t) p_k^{(\gamma,\delta)}(1) .$$
The $(\mathbf{N},p_n)$ means of the random Fourier--Jacobi series (\ref{1.11.2}) at the point $y=1$ is
\begin{equation}\label{5.7}
\mathbf{H}_n^{(\gamma,\delta)}(f,1,\omega):=\frac{1}{\mathrm{P}_n}\sum_{k=0}^n p_{n-k}\mathbf{S}_k^{(\gamma,\delta)}(f,1,\omega).
\end{equation}
The integral form of $(\mathbf{N},p_n)$ means (\ref{5.7}) is
\begin{eqnarray*}
\mathbf{H}_n^{(\gamma,\delta)}(f,1,\omega)&:=&\frac{1}{\mathrm{P}_n}\sum_{k=0}^n p_{n-k}\mathbf{S}_n^{(\gamma,\delta)}(f,1,\omega)\\
&=& \int_{-1}^1 \mathbf{h}_n^{(\gamma,\delta)}(f,1,t) \rho^{(\gamma,\delta)}(t)dX(t,\omega),
\end{eqnarray*}
where $$\mathbf{h}_n^{(\gamma,\delta)}(f,1,t):=\frac{1}{\mathrm{P}_n}\sum\limits_{k=0}^n p_{n-k} \mathbf{s}_n^{(\gamma,\delta)}(f,1,t) .$$
By Lemma \ref{l1}, for $\alpha=1,$
\begin{eqnarray*}
&&P\Big(\Big|\mathbf{H}_n^{(\gamma,\delta)}(f,1,\omega)- \mathbf{H}_m^{(\gamma,\delta)}(f,1,\omega)\Big|>\epsilon\Big)\\
&\leq& \frac{C2^{\alpha+1}}{\epsilon'^\alpha(\alpha+1)} \int_{-1}^1 \Bigg|\Big(\mathbf{h}_n^{(\gamma,\delta)}(f,1,t)- \mathbf{h}_m^{(\gamma,\delta)}(f,1,t)\Big)\rho^{(\gamma,\delta)}(t)\Bigg|dt.
\end{eqnarray*}
If the equations (\ref{5.13.1}), (\ref{5.12}) and (\ref{5.12.1}) are satisfied by $\mathrm{P}_n$ with the antipole condition (\ref{5.13}), then
the N{\"o}rlund $(\mathbf{N},p_n)$ sum of Fourier--Jacobi series (\ref{1.1}) at the point $y=1$ converges to $A.$
So, $\{ \mathbf{h}_n^{(\gamma,\delta)}(f,1,t) \}$ forms a Cauchy sequence.
Thus
\begin{equation*}
\lim_{n \rightarrow \infty} \int_{-1}^1 \Bigg|\bigg(\mathbf{h}_n^{(\gamma,\delta)}(f,1,t)- \mathbf{h}_m^{(\gamma,\delta)}(f,1,t)\bigg)\rho^{(\gamma,\delta)}(t)\Bigg|dt=0.
\end{equation*}
This gives that $\mathbf{H}_n^{(\gamma,\delta)}(f,1,\omega)$ is a Cauchy sequence and converges in probability to a random variable.
With the help of Lemma \ref{l1},
\begin{eqnarray*}
&&P\Bigg(\Big|\mathbf{H}_n^{(\gamma,\delta)}(f,1,\omega)-\int_{-1}^1 A \rho^{(\gamma,\delta)}(\psi)dX(\psi,\omega)\Big|>\epsilon\Bigg)\\
&\leq& \frac{C2^{\alpha+1}}{\epsilon^\alpha(\alpha+1)} \int_{-1}^1 \Big|\Big(\mathbf{h}_n^{(\gamma,\delta)}(f,1,t)- A\Big)\rho^{(\gamma,\delta)}(t)\Big|dt.
\end{eqnarray*}
Hence by the Theorem \ref{t.1},
\begin{equation*}
\lim_{n \rightarrow \infty} \int_{-1}^1 \Big|\Big(\mathbf{h}_n^{(\gamma,\delta)}(f,1,t)- A\Big)\rho^{(\gamma,\delta)}(t)\Big|dt=0.
\end{equation*}
This completes the proof of theorem.
\end{proof}
The following two theorems state about the generalized N{\"o}rlund summability and lower triangular summability of random Fourier--Jacobi series (\ref{1.11.2})
which can be established by following the same steps in Theorem \ref{5T.1}.
\begin{theorem}\label{5T.3}
If $\gamma \geq 0,\; \delta-\gamma >1,\; \delta+\gamma > 0$ and $\{q_n\},\; \{r_n\}$ are the non--negative real constants sequences satisfy the conditions (\ref{1.15}), (\ref{1.16}) and the antipole condition (\ref{1.17}) stated in the Theorem \ref{5t.1},
 then the random Fourier--Jacobi series (\ref{1.11.2}) is $|\mathbf{N},q_n,r_n|$
 summable in probability to the stochastic integral
(\ref{5.3})
at the point $y=1.$
\end{theorem}

\begin{theorem}\label{5T.2}
If $0 \leq \gamma \leq 1/2,\; \delta \geq 0$ and the infinite triangular regular matrix $M=(\mathbf{b}_{n,k})$ be such that, its entries $(\mathbf{b}_{n,k})$ are positive, monotonic increasing in $k, \;0 \leq k \leq n,$ and satisfy the conditions (\ref{1.18}), (\ref{1.19}), (\ref{1.20}), (\ref{1.21}) with the antipole condition (\ref{1.22}) stated in the Theorem \ref{5t.2},
then the random Fourier--Jacobi series (\ref{1.11.2}) is $(\mathbf{T})$ summable to the stochastic integral
(\ref{5.3}) at the point $y=1.$
\end{theorem}
\section*{Acknowledgments}
This research work was supported by University Grant Commission (National Fellowship with letter no-F./2015-16/NFO-2015-17-OBC-ORI-33062).

\end{document}